\documentclass[12pt]{article}
\usepackage[english]{babel}
\usepackage{amsmath, amsthm, amscd, amsfonts, amssymb, graphicx, xcolor, algorithm, algorithmic,setspace}
\usepackage[a4paper, total={7in, 9in}]{geometry}

\usepackage{subfigure}
\usepackage{amssymb}
\usepackage{siunitx}
\PassOptionsToPackage{hyphens}{url}
\usepackage[colorlinks=true]{hyperref}
\usepackage[utf8]{inputenc}
\usepackage{booktabs}
\usepackage{longtable}
\usepackage{adjustbox}
\usepackage{array}
\usepackage{url}
\usepackage{cite}
\usepackage{titlesec}
\usepackage{bm}
\usepackage{authblk}
\usepackage{xcolor} 
\newtheorem{theorem}{Theorem}[section]
\newtheorem{lemma}[theorem]{Lemma}
\newtheorem{proposition}[theorem]{Proposition}
\newtheorem{corollary}[theorem]{Corollary}

\theoremstyle{definition}
\newtheorem{example}[theorem]{Example}
\numberwithin{equation}{section}

\newcommand{\set}[1]{\left\{#1\right\}}
\newcommand{\paren}[1]{\left(#1\right)}

\newcommand{\R}{{\mathbb R}}

\newcommand{\inner}[2]{\left\langle #1, #2 \right\rangle}

\author[1]{Huiyuan Guo}
\author[2]{Juan Jos\'e Maul\'en}
\author[1]{Juan Peypouquet}

\affil[1]{Bernoulli Institute, University of Groningen, the Netherlands}
\affil[2]{Center for Mathematical Modeling, University of Chile, Chile}

\title{A Speed Restart Scheme for a Dynamical System with Hessian-Driven Damping and Three Constant Coefficients}

\begin{document}
	
	\maketitle
	
	\begin{abstract}
		\setlength{\parindent}{0pt}
		
		\noindent In this paper, we study a speed restart scheme for an inertial system with Hessian-driven damping. We establish a linear convergence rate for the function values along the restarted trajectories without assuming the strong convexity of the objective function. Our numerical experiments show improvements in the convergence rates, both for the continuous-time dynamics, and when applied to inertial algorithms as a heuristic. \\
		
		\textbf{Keywords:} Convex optimization; inertial methods; speed restart. \\
		
		\textbf{MSC2020:} 37N40 $\cdot$ 90C25 $\cdot$ 65K10 (primary). 34A12 (secondary).
		
	\end{abstract}
	
	\section{Introduction}
	
	The use of inertial methods is a popular first-order approach to smooth convex optimization problems. Their design stems from the seminal work of Polyak 
	\cite{polyak1964some} (see also \cite{Polyak1987}) on the numerical approximation of minimizers of a smooth and strongly convex function $f$. The {\it Heavy Ball Method}, defined by iterations of the form
	\begin{equation}\tag{HBM}\label{eq:HBM}
		x_{k+1}=x_k+\theta(x_k-x_{k-1})-h\nabla f(x_k),\qquad \theta>0,
	\end{equation}
	is a variant of {\it gradient descent} motivated in \cite{polyak1964some} as a finite difference discretization of the {\it Heavy Ball with Friction} dynamics, namely
	\begin{equation}\tag{HBF}\label{eq:HBF}
		\ddot{x}(t)+\alpha \dot{x}(t)+\nabla f(x(t))=0, \qquad\alpha>0,
	\end{equation}
	when $\theta\sim 1-\alpha h$. Naturally, the coefficients have a major impact on the behavior of the solutions of \eqref{eq:HBF} and the sequences generated by \eqref{eq:HBM}, and they must be properly tuned for good performance. The (non-strongly) convex case was further studied in \cite{Alvarez2000}
	. \\
	
	In \cite{N1983}, Nesterov proposed a new variant of gradient descent, which is similar$-$but different$-$to \eqref{eq:HBM}, comprising two subiterations
	$$
	y_k=x_k+\theta_k(x_k-x_{k-1}),\qquad x_{k+1}=y_k-h\nabla f(y_k),\qquad \theta_k>0,
	$$
	and oriented to minimizing (non-strongly) convex functions. Unlike in \eqref{eq:HBM}, the parameters are not constant, but vary along the iterations in a precise fashion. It was later shown in \cite{SBC16} that Nesterov's method can be interpreted as a discretization of the inertial system with {\it Asymptotic Vanishing Damping}, namely
	\begin{equation}\tag{AVD}\label{eq:AVD}
		\ddot{x}(t)+\frac{\alpha}{t}\dot{x}(t)+\nabla f(x(t))=0,    \qquad \alpha>0.
	\end{equation}
	A constant-parameter version of Nesterov's method was proposed in \cite{N2004} for strongly convex functions. It can also be interpreted as a discretization of \eqref{eq:HBF}, different from the one giving \eqref{eq:HBM}. \\
	
	Algorithms designed upon the continuous-time models \eqref{eq:HBF} and \eqref{eq:AVD} often present strong oscillations, an undesirable property in view of the high computational cost associated with frequent evaluations of the objective function, especially in high dimensions, if one wanted to select the {\it best iterate up to iteration $k$}. In order to reduce the oscillatory behavior frequently experienced by \eqref{eq:HBF}, a {\it Damped Inertial Newton-like} system, given by
	\begin{equation}\tag{DIN}\label{eq:din_alpha_beta}
		\ddot{x}(t)+\alpha\dot{x}(t)+\beta\nabla^{2} f(x(t))\dot{x}(t)+\nabla f(x(t))=0,    
	\end{equation}
	was proposed and analyzed in \cite{AABR2002}. A {\it Hessian-driven damping} term, whose inspiration comes from the {\it Continuous Newton} method of \cite{Alvarez1998}, is added to \eqref{eq:HBF}. A similar approach was used in \cite{APR16}, where the authors study the system governed by
	\begin{equation}\tag{DIN-AVD}\label{eq:DIN-AVD}
		\ddot{x}(t)+\frac{\alpha}{t}\dot{x}(t)+\beta\nabla^{2} f(x(t))\dot{x}(t)+\nabla f(x(t))=0, \qquad  \alpha,\beta>0,
	\end{equation}
	in order to mitigate the oscillations that are typical of \eqref{eq:AVD}. In this work, we study a {\it Weighted Inertial Newton-like} system, given by
	\begin{equation}\tag{WIN}\label{ORI_intr}
		\ddot{x}(t)+\alpha\dot{x}(t)+\beta\nabla^{2} f(x(t))\dot{x}(t)+\gamma\nabla f(x(t))=0,
	\end{equation}
	where $\alpha>0$, $\beta\geq 0$ and $\gamma>0$. Although \eqref{ORI_intr} can be reduced to \eqref{eq:din_alpha_beta} by a reparameterization of time \cite{ACR19,ABCR22,ACR2019} and a renaming of the coefficients, we prefer to keep the explicit dependence on $\gamma$ for explainability purposes, and for ease of translation into algorithmic implementations.

	\subsubsection*{A simple but illustrative case study}
	Pick $\rho>1$, and define the function $f:\R^n\to\R$ by 
	\begin{equation} \label{E:example_function}
		f(x)=\frac{1}{2}(x_1^2 + \rho x_2^2 + \rho^2 x_3^2 + \ldots+\rho^{n-1}x_n^2),    
	\end{equation}
	which is quadratic, $1$-strongly convex and $\rho^{n-1}$-smooth. The differential equations described above are linear, but become worse conditioned as $\rho$ becomes larger. Moreover, certain combinations of the coefficients will produce oscillatory solutions in the inertial systems. This will be the case, for instance, if 
	\begin{equation}\label{eq:gamma_example}
		\gamma = \dfrac{1}{4\rho^i}(\alpha+\rho^i\beta)^2 + \varepsilon, \qquad \hbox{for some $i\in\{0,\dots,n-1\}$ and $\varepsilon>0$},
	\end{equation}
	in \eqref{ORI_intr}.
	However, these oscillations appear considerably less pronounced when the Hessian-driven damping term is present, as shown in Figure \ref{DINabr}. Concerning the comparison with \eqref{eq:DIN-AVD}, Figure \ref{fig:comp_DINAVD} exhibits a noticeable improvement when using a constant coefficient for $\dot x$ instead of a vanishing one, with the former producing a more pronounced slope in logarithmic scale, which translates into a better linear rate. Since $f$ is strongly convex, this would be expected if one adapted the said coefficient to the strong convexity parameter (common practice would dictate $\alpha<2$ in both cases). Instead, we have used $\alpha=3$, which is the typical choice corresponding to \eqref{eq:AVD} and \eqref{eq:DIN-AVD} (see, for instance, \cite{SBC16,APR16,ACPR18,attouch2016rate}).

		\begin{figure}[htbp]
			\centering
			\subfigure[$\rho=10$]{\includegraphics[width=0.48\textwidth]{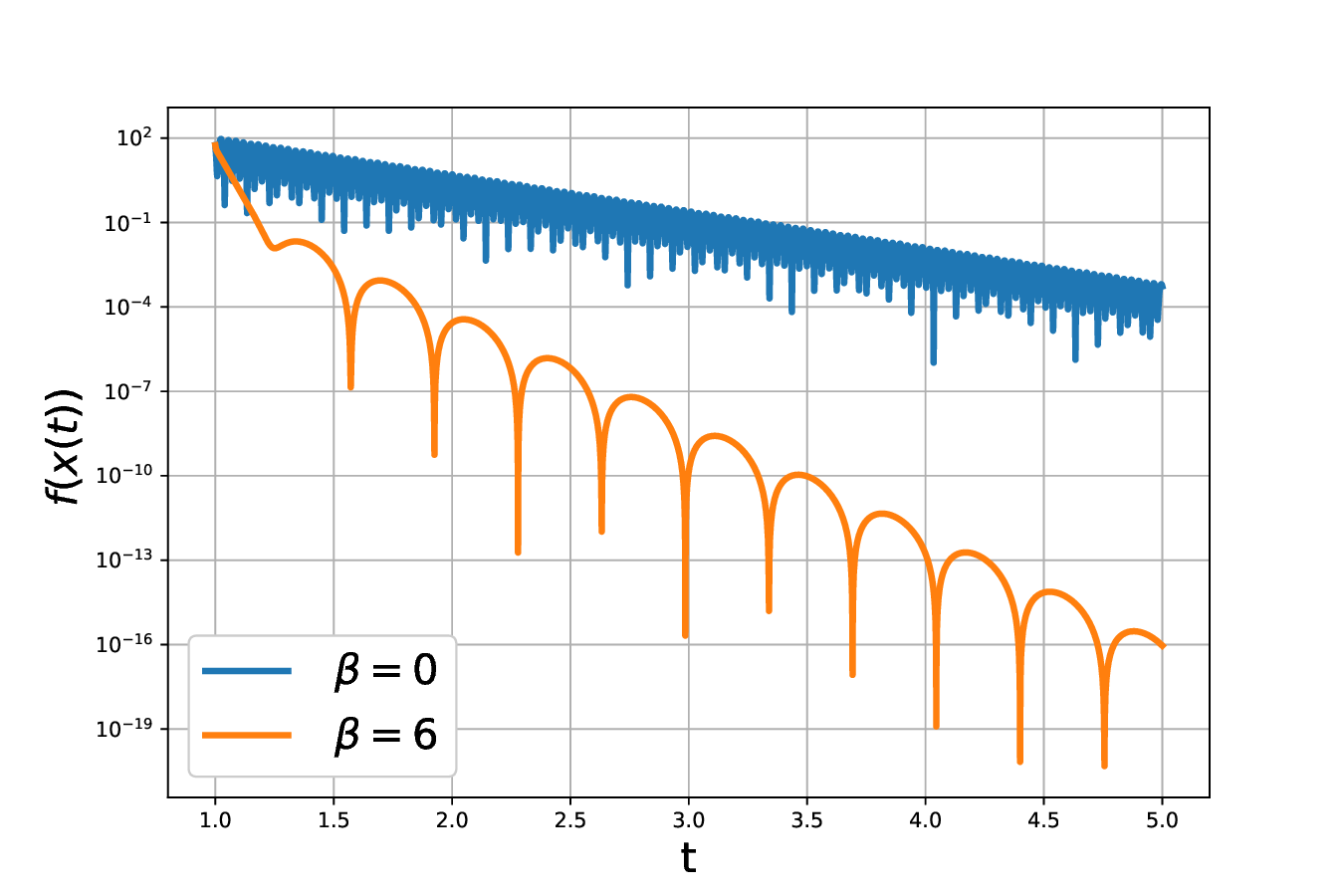}}
			\subfigure[$\rho=100$]{\includegraphics[width=0.48\textwidth]{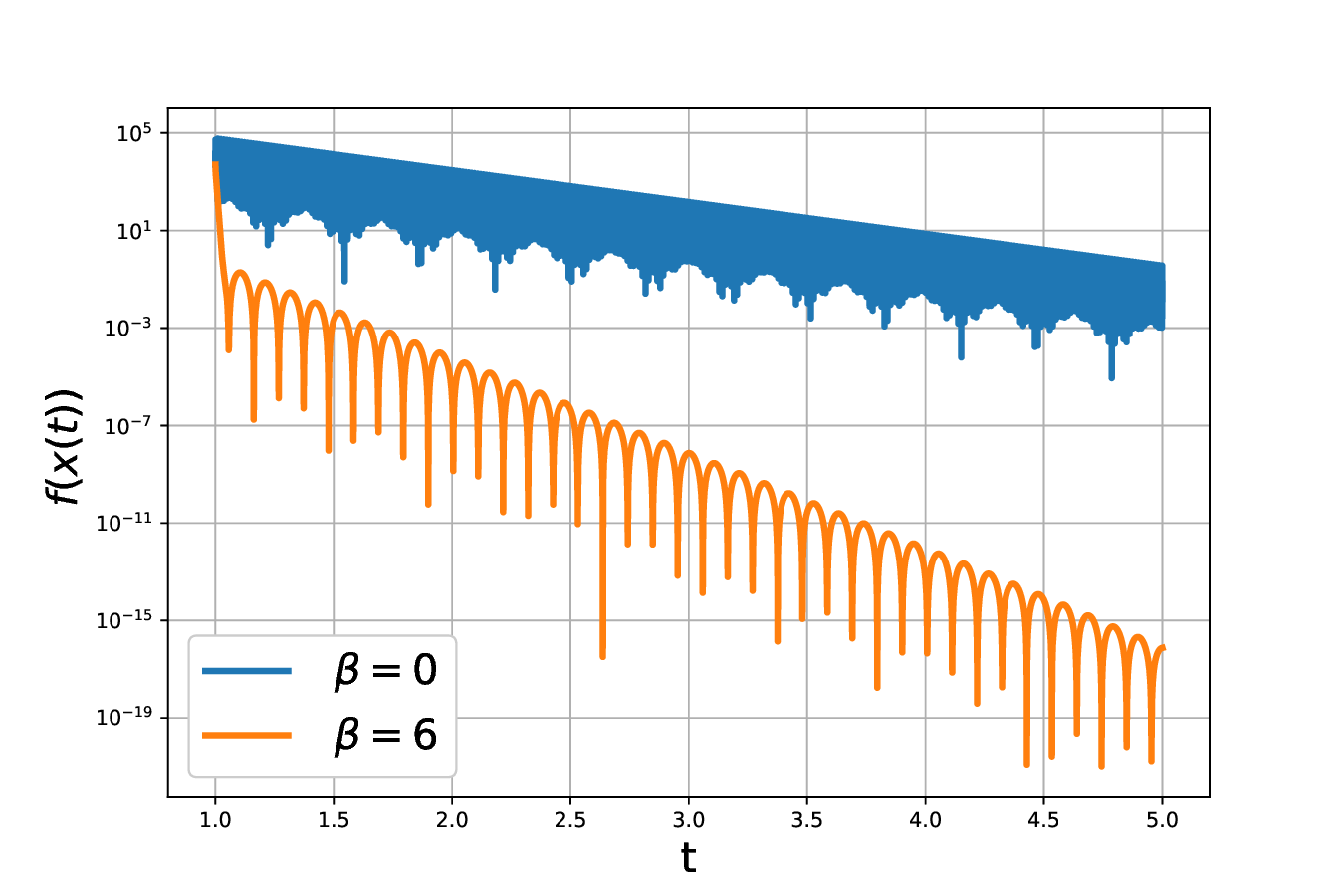}}
			\caption{Depiction of the function values on the interval [1,5] for $\alpha=3$, $\gamma$ as in \eqref{eq:gamma_example} with $i=1$ and $\varepsilon=0.1$, with and without the Hessian term.}
			\label{DINabr}
		\end{figure}
		
		\begin{figure}[htbp]
			\centering
			\subfigure[$\rho=10$]{\includegraphics[width=0.48\textwidth]{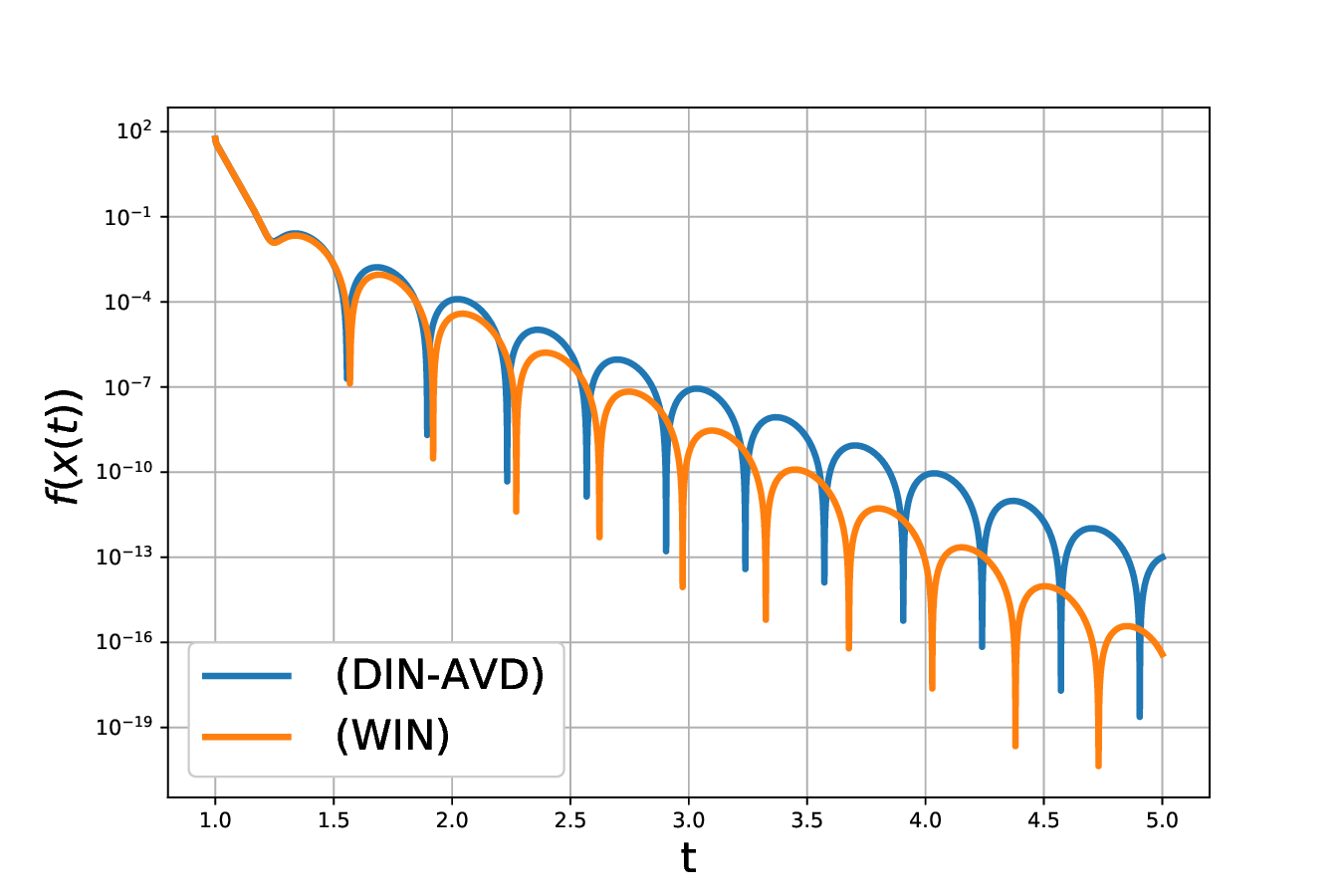}}
			\subfigure[$\rho=100$]{\includegraphics[width=0.48\textwidth]{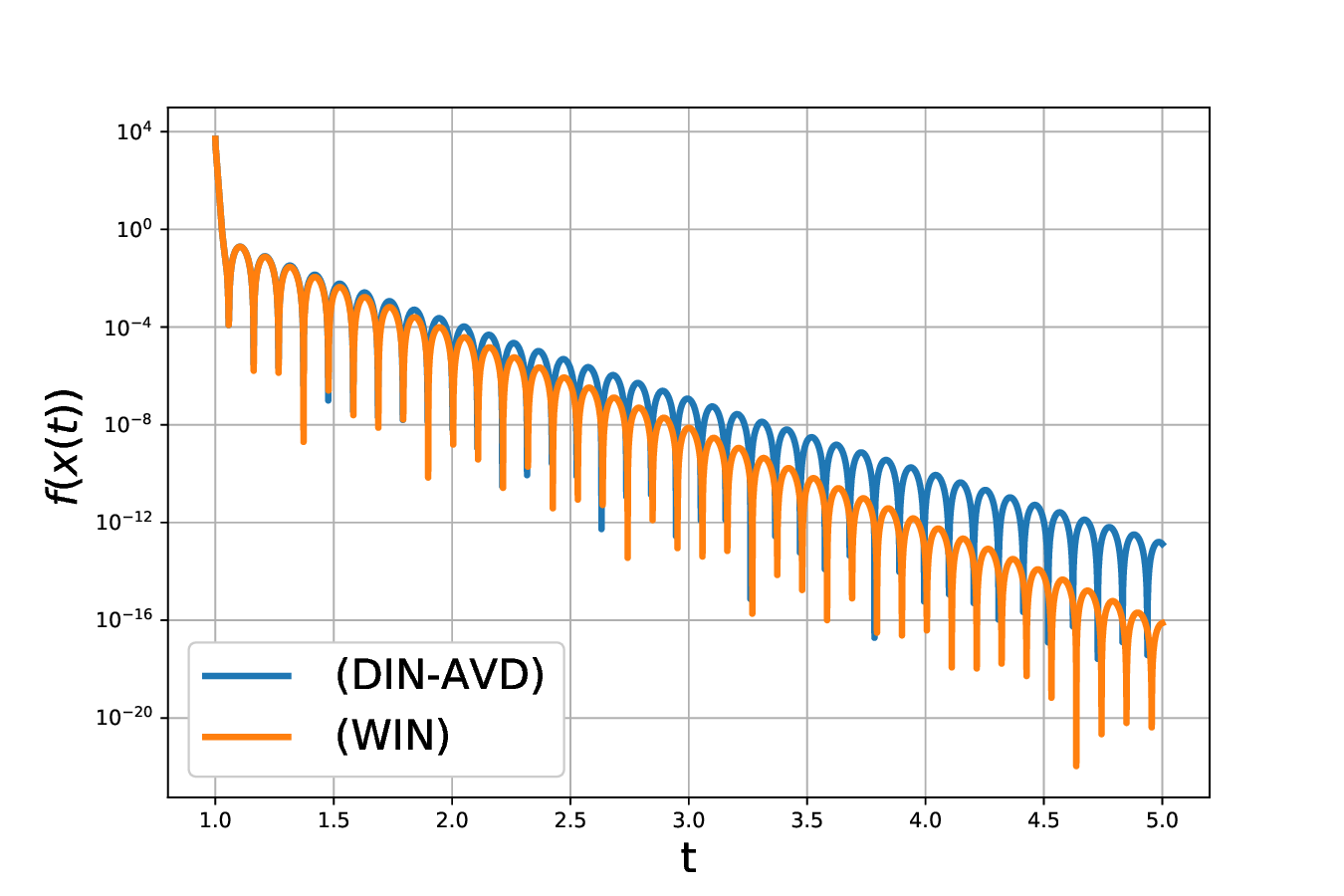}}
			\caption{Comparison of function values of \eqref{ORI_intr} and \eqref{eq:DIN-AVD} for $\alpha=3$, $\beta=6$, $\gamma$ as in  \eqref{eq:gamma_example} with $\varepsilon=1$.}
			\label{fig:comp_DINAVD}
		\end{figure}
		The better performance of the system with constant coefficients and Hessian-driven damping \eqref{ORI_intr}, when applied to the example above, is an evidence of its potential as a continuous-time model for better performing algorithms.
		
		\subsubsection*{Restarting strategies}
		
		Restarting techniques represent an alternative way to speed up inertial methods by reducing the oscillations. The general idea is that, when some criterion is met, the current state of the system (the current iterate of the algorithm) is used as the initial condition to run a new cycle. A classical strategy introduced in \cite{N2013} (see also \cite{NN85}) for the accelerated gradient method is to restart the algorithm at fixed intervals, which depend on the strong convexity parameter of the function, which might be unknown. This difficulty has been addressed, for instance, in \cite{NNG19,AKL23,LX15,roulet2017sharpness,renegar2022simple}. \\
		
		Two heuristic adaptive policies were proposed in \cite{OC13,GB14}. In the first case, the algorithm is restarted if the value of the objective function at the next iterate will be higher than the value at the current one. In the second one, the algorithm is restarted if the vector that indicates the next movement will form an acute angle with the gradient at the current point. In both cases, the objective function values decrease along the iterations, and both correspond to
		$$0<\frac{d}{dt}\left[f\big(x(t)\big)\right]=\left\langle \nabla f\big(x(t)\big),\dot x(t)\right\rangle,$$
		in continuous-time models. Although these schemes show remarkable performance in numerical experiments, the theoretical analysis for the convergence rate has not been established yet. Implementation of restarts has been studied for different classes of algorithms, such as FISTA \cite{FQ19, AKL19, AKL192,ADLA22}, primal-dual splitting algorithms \cite{applegate2023faster}, Schwarz methods \cite{park2021accelerated}, Stochastic gradient descent \cite{wang2022scheduled}, and first order methods for nonconvex optimization \cite{li2023restarted}. \\
		
		In \cite{SBC16}, a {\it speed restart} strategy for \eqref{eq:AVD} is analyzed, and linear convergence of the objective function values is established, in the strongly convex case. The main idea is to restart the dynamics at the point where the speed ceases to increase. This restarting scheme is then implemented on Nesterov's accelerated gradient method as a heuristic. Although this does not beat the ones described above, it does have the theoretical support of the analysis in continuous time. Analogue results for \eqref{eq:DIN-AVD} were obtained in \cite{MP23}. The authors report a 34.67\% increase for the constant $B$ in the aproximation $f\big(x(t)\big)-\min(f)\sim Ae^{-Bt}$ when the restarting scheme is applied to \eqref{eq:DIN-AVD}, with respect to \eqref{eq:AVD}.\footnote{Also, the constant $A$ is $4.6\times 10^{5}$ times smaller. This is not mentioned in \cite{MP23}, but can be easily computed.}

		\subsubsection*{Our contribution}

		In line with the discussion above, the aim of this paper is to analyze the impact of the speed restarting scheme on the dynamical system \eqref{ORI_intr}, in order to establish theoretical foundations to accelerate inertial algorithms by means of a speed restarting policy. By considering general parameters $\alpha>0$, $\beta\ge 0$ and $\gamma>0$, we can encompass those algorithms which do not involve a Hessian-driven damping term, such as the Heavy Ball method \cite{Polyak1987,polyak1964some}, as well as those that do, which include Nesterov's acclerated gradient algorithm \cite{N1983, N2004} and other optimized gradient methods \cite{Drori_2014,Kim_2016,Kim_2017,Park_2023}. We establish the linear convergence of the function values under a Polyak-\L{}ojasiewicz inequality, and show how the speed restart improves the convergence rates of the solutions of \eqref{ORI_intr}. \\

		
		The paper is organized as follows: In Section \ref{sec:srs}, we describe the speed restart scheme, along with the restarted trajectories for \eqref{ORI_intr}, and present our main theoretical result, which established the linear convergence of the function values on the restarted trajectory to the optimal value of the problem. The technical details are collected in Section \ref{sec:lemmas}. The most relevant are, on the one hand, the upper and lower bounds for the restarting times and, on the other, an estimation of the function value decrease between restarts. Finally, although this is not the main purpose of this work, we present some numerical experiments in Section \ref{sec:numerical} to illustrate how the speed restart scheme improves the convergence rate of the trajectories of \eqref{eq:HBF} and \eqref{ORI_intr}, and how it can enhance the performance of the corresponding algorithms.
		
		\section{Speed Restart Scheme}
		\label{sec:srs}
		
		Throughout this paper, let $\mathcal{H}$ be a Hilbert space, and let $f:\mathcal{H}\rightarrow\mathbb{R}$ be a convex function of class $\mathcal{C}^{2}$, which attains its minimum value $f^{*}$. Also, assume that $f$ satisfies the Polyak-\L{}ojasiewicz inequality
		\begin{equation}\label{PL}
			2\mu(f(z)-f^{*})\le \|\nabla f(z)\|^{2}
		\end{equation}
		for all $z\in \mathcal{H}$ and some $\mu>0$, and that $\nabla f$ is Lipschitz-continuous with constant $L>0$. \\
		
		Consider the inertial dynamical system
		\begin{equation}\label{ORI}
			\ddot{x}(t)+\alpha \dot{x}(t)+\beta\nabla^{2}f(x(t))\dot{x}(t)+\gamma \nabla f(x(t))=0,
		\end{equation}
		with parameters $\alpha>0$, $\beta\geq 0$ and $\gamma>0$.
		Given $z\in \mathcal{H}$, let $x_z$ be the solution of \eqref{ORI} with initial conditions $x_z(0)=z$, $\dot{x}_z(0)=0$. The {\it speed restart time} for $x_z$ is 
		\begin{equation} \label{E:restart_time}
			T(z)=\inf \left\{t>0:\frac{d}{dt}\|\dot{x}_z(t)\|^{2}\leq0\right\}.
		\end{equation}
		
		This definition does not directly imply that the restart will ever occur. However, Corollary \ref{infc} and Proposition \ref{upper bound} below explicitly provide positive numbers $\tau_*$ and $\tau^*$ such that
		\begin{equation}\label{eq:supinf}
			\tau_*\le T(z)\le \tau^*\qquad\hbox{for every }z\notin\operatorname{argmin}(f).
		\end{equation}
		
		Now, for $t\in(0,T(z))$, we have
		\begin{eqnarray} 
			\dfrac{d}{dt}f(x_z(t)) &=& \langle \nabla f(x_z(t)),\dot{x}_z(t) \rangle \nonumber \\
			&=&-\dfrac{1}{\gamma}\langle \ddot{\bm{x}}_z(t),\dot{x}_z(t)\rangle-\dfrac{\alpha}{\gamma}\|\dot{x}_z(t)\|^{2}-\dfrac{\beta}{\gamma}\langle \nabla^{2}f(x_z(t))\dot{x}_z(t),\dot{x}_z(t)\rangle \label{increase_0} \\
			&\le & -\frac{\alpha}{\gamma}\|\dot{x}_z(t)\|^{2}, \label{increase}
		\end{eqnarray}
		because $\langle \nabla^{2}f(x_z(t))\dot{x}_z(t),\dot{x}_z(t)\rangle\geq0$ by the convexity of $f$, and $\langle \ddot{x}_z(t),\dot{x}_z(t)\rangle\geq0$  by the definition of $T(z)$. Therefore, the function $t\rightarrow f(x_z(t))$ can only start increasing {\em after} the speed restart time. \\
		
		By appropriately bounding the speed from below, it is possible to quantify the reduction in the function value gap from the initial time to the speed restart time. More precisely, Proposition \ref{fdecrease} gives an explicit constant $Q\in(0,1)$ such that
		\begin{equation} \label{fdecrease}
			f\big(x_z(T(z))\big)\le Q\big(f(z)-f^*\big).
		\end{equation}
		
		Given $z\in \mathcal{H}$, the \textit{restarted trajectory} $\bm{x}_{z}:[0,+\infty)\rightarrow \mathcal{H}$ is defined by glueing together pieces of solutions of \eqref{ORI}, as follows:
		\begin{enumerate}
			\item First, compute $x_{z}$, $T_{1}=T(z)$ and $S_{1}=T_{1}$, and define $\bm{x}_{z}(t)=x_{z}(t)$ for $t\in [0,S_{1}]$.
			\item For $i\geq 1$, having defined $\bm{x}_{z}(t)$ for $t\in [0,S_{i}]$, set $z_{i}=\bm{x}_{z}(S_{i})$, and compute $x_{z_i}$. Then, set $T_{i+1}=T(z_{i})$ and $S_{i+1}=S_{i}+T_{i+1}$, and define $\bm{x}_{z}(t)=x_{z_i}(t-S_{i})$ for $t\in (S_{i},S_{i+1}]$.
		\end{enumerate}
		
		Condition \eqref{eq:supinf} ensures that $S_i$ is well defined for all $i\geq 1$. The resulting trajectory $\bm{x}_z$ is continuous and piecewise continuously differentiable. By \eqref{increase}, we have:
		
		\begin{proposition}\label{trajde}
			The function $t\mapsto f(\bm{x}_{z}(t))$ is nonincreasing.
		\end{proposition}
		
		Inequalities \eqref{eq:supinf} and \eqref{fdecrease}, together with Propositoin \ref{trajde}, hold the key to our main theoretical result, which establishes the linear convergence of $f(\bm{x}_{z}(t))$ to $f^{*}$:
		
		\begin{theorem}\label{teo}
			Let $f:\mathcal{H}\rightarrow\mathbb{R}$ be a convex function of class $\mathcal{C}^{2}$, which attains its minimum value $f^{*}$. Assume also that $f$ satisfies the Polyak-\L{}ojasiewicz inequality \eqref{PL} with $\mu>0$, and that $\nabla f$ is Lipschitz continuous with constant $L>0$. Given $\alpha>0$, $\beta\geq 0$ and $\gamma>0$, there exist constants $C,K>0$ such that, for every initial point $z\in \mathcal{H}$, the restarted trajectory $\bm{x}_{z}$ satisfies
			\[f(\bm{x}_{z}(t))-f^{*}\leq C e^{-Kt}(f(z)-f^{*})\leq\frac{CL}{2}e^{-Kt}\operatorname{dist}(z,\operatorname{argmin}(f))^{2},\]
			for all $t>0$.
		\end{theorem}
		
		\begin{proof}
			Let $m$ be the largest positive integer such that $m\tau^{*}\leq t$. By time $t$, the trajectory will have been restarted at least $m$ times. From Proposition \ref{trajde}, we know that
			\[f(\bm{x}_{z}(t))\leq f(\bm{x}_{z}(m\tau^{*}))\leq f(\bm{x}_{z}(m\tau_{*})). \]
			We use \eqref{fdecrease} inductively to obtain
			\[f(\bm{x}_{z}(t))-f^{*}\leq Q^{m}(f(z)-f^{*}).\]
			By definition, $(m+1)\tau^{*}>t$, which entails $m>\frac{t}{\tau^{*}}-1$. Since $Q\in (0,1)$, we have 
			\[Q^{m}\leq Q^{\frac{t}{\tau^{*}}-1}=\frac{1}{Q}\operatorname{exp}\left(\frac{\operatorname{ln}(Q)}{\tau^{*}}t\right).\]
			We obtain the first inequality by setting $C=Q^{-1}>0$ and $K=-\frac{1}{\tau^{*}}\operatorname{ln}(Q)>0$. The second one follows from the fact that $\nabla f$ is Lipschitz-continuous with constant $L$.
		\end{proof}
		
		The values of $C$ and $K$ are given by Corollary \ref{infc} and Propositions \ref{upper bound} and \ref{fdecrease}, whose proofs are technical. In the next section, we provide all the relevant computations and discuss how the obtained results compare to previous works in the literature. \\
		
		From the proof of Theorem \ref{teo}, we see that the tightness of the approximation of the speed restart time given by $\tau_*$ and $\tau^*$ has a profound effect on the quality of the estimation of the convergence rate of $f(\bm{x}_{z}(t))$ to $f^{*}$.
		
		\section{Evolution of the system between restarts} \label{sec:lemmas}
		
		The proof of the main result is technical and will be split into several lemmas. In order to lighten the notation, given an initial condition $z \in \mathcal{H}$, we simply denote by $x$ the solution of \eqref{ORI} with initial conditions $x(0)=z$, and $\dot{x}(0)=0$. 
		
		\subsection{Preliminary estimations}
		We first define some functions that will be useful in the forthcoming analysis. Equation (\ref{ORI}) can be rewritten as
		\begin{equation}\label{Rewrite}
			\frac{d}{dt}(e^{\alpha t}\dot{x}(t))=-\gamma e^{\alpha t}\nabla f(x(t))-\beta e^{\alpha t}\nabla^{2}f(x(t))\dot{x}(t).
		\end{equation}
		Integrating (\ref{Rewrite}) over $\left[0,t\right]$, we get

		\begin{align}\label{Int}
			e^{\alpha t}\dot{x}(t)&=-\gamma \int_{0}^{t}e^{\alpha u}\nabla f(x(u))du-\beta\int_{0}^{t}e^{\alpha u}\nabla^{2}f(x(u))\dot{x}(u) du \nonumber \\
			&=-\gamma\int_{0}^{t}e^{\alpha u}(\nabla f(x(u))-\nabla f(z))du-\beta\int_{0}^{t}e^{\alpha u}\nabla^{2}f(x(u))\dot{x}(u)du \nonumber\\
			&\quad
			-\frac{\gamma (e^{\alpha t}-1)}{\alpha}\nabla f(z).
		\end{align}
		We define the two integrals obtained as
		\begin{equation}\label{IJe}
			I_{z}(t)=\gamma \int_{0}^{t}e^{\alpha u}(\nabla f(x(u))-\nabla f(z))du\quad \text{and}\quad J_{z}(t)=\beta\int_{0}^{t}e^{\alpha u}\nabla^{2}f(x(u))\dot{x}(u) du.
		\end{equation}
		In order to majorize $I_{z}(t)$ and $J_{z}(t)$, we define the function
		\begin{equation}
			M_{z}(t)=\sup\limits_{u\in \left (0,t \right]}\left[\frac{\|\dot{x}(u)\|}{1-e^{-\alpha u}}\right],
		\end{equation}
		which is positive, non-decreasing and continuous on $(0,T(z))$.	
		\begin{lemma}\label{IJ}
			For every $t>0$, we have
			\begin{eqnarray*}
				\|\nabla f(x(t))-\nabla f(z)\| 
				& \leq & L M_{z}(t)\left(t+\frac{1}{\alpha}e^{-\alpha t}-\frac{1}{\alpha}\right), \\
				\|I_{z}(t)\| & \leq & \gamma L M_{z}(t)\left(\frac{1}{\alpha}t e^{\alpha t}-\frac{2}{\alpha^{2}}e^{\alpha t}+\frac{2}{\alpha^{2}}+\frac{t}{\alpha}\right), \\
				\|\nabla^{2}f(x(t))\dot{x}(t)\| 
				& \leq &  LM_{z}(t)\big(1-e^{-\alpha t}\big) \\
				\|J_{z}(t)\|
				& \leq & \beta L M_{z}(t)\left(\frac{1}{\alpha} e^{\alpha t}-\frac{1}{\alpha}-t\right).
			\end{eqnarray*}  
		\end{lemma}
		
		\begin{proof}
			For the first estimation, we use the Lipschitz-continuity of $\nabla f$ and the fact that $M_{z}$ is non-decreasing to obtain
			\[
			\|\nabla f(x(u))-\nabla f(z)\|\leq L \left\|\int_{0}^{u}\dot{x}(s)ds\right\| \leq L M_{z}(u)\int_{0}^{u}(1-e^{-\alpha s})ds=L M_{z}(u)\paren{u+\frac{1}{\alpha}e^{-\alpha u}-\frac{1}{\alpha}}.
			\]
			As a consequence, we get
			\[\|I_{z}(t)\|\leq \gamma L M_{z}(t) \int_{0}^{t}e^{\alpha u}\left(u+\frac{1}{\alpha}e^{-\alpha u}-\frac{1}{\alpha}\right)du=\gamma L M_{z}(t) \paren{\frac{1}{\alpha}t e^{\alpha t}-\frac{2}{\alpha^{2}}e^{\alpha t}+\frac{2}{\alpha^{2}}+\frac{t}{\alpha}}.\]       
			
			For the second inequality, we first estimate	 
			\begin{align*} 	
				\|\nabla^{2}f(x(u))\dot{x}(u)\|&=\left\|\lim_{r\rightarrow u}\frac{\nabla f(x(r))-\nabla f(x(u))}{r-u}\right\|\\
				&\leq \lim_{r\rightarrow u}\frac{L}{r-u}\int_{u}^{r}\|\dot{x}(s)\|ds \\
				&\leq \lim_{r\rightarrow u}\frac{LM_{z}(r)}{r-u}\int_{u}^{r}\big(1-e^{-\alpha s}\big)\,ds \\
				&= LM_{z}(u)(1-e^{-\alpha u}),
			\end{align*}
			and then conclude that
			\[ \|J_{z}(t) \| \leq \beta L \int_{0}^{t} e^{\alpha u}M_{z}(u)(1-e^{-\alpha u})du\leq\beta LM_{z}(t)\paren{\frac{1}{\alpha}e^{\alpha t}-\frac{1}{\alpha}-t},
			\]
			as claimed.
		\end{proof}
		
		The function $M_{z}(t)$, can be majorized in such way that the dependence of the bound on the initial condition $z$ is only given by $\nabla f(z)$. For this purpose, we define the function
		\begin{equation}\label{H}
			H(t)=1+\frac{2L\gamma}{\alpha^{2}}-\frac{L\beta}{\alpha}-\frac{Lt(\frac{\gamma}{\alpha}e^{\alpha t}+\frac{\gamma}{\alpha}-\beta)}{e^{\alpha t}-1}.
		\end{equation}
		
		\begin{proposition}\label{function:H(t)}
			The function $H$ is decreasing on $(0,\infty)$. 
		\end{proposition}
		
		\begin{proof}
			The derivative of $H$ is given by $H'(t)=(e^{\alpha t}-1)^{-2}h(t)$, 
			where
			\[h(t)=-\frac{L\gamma}{\alpha}e^{2\alpha t}+\frac{L\gamma}{\alpha}+L\beta e^{\alpha t}-L\beta+2L\gamma t e^{\alpha t} - L\alpha \beta t e^{\alpha t}.\]
			As a consequence, $h(0)=0$ and
			\[h'(t)=2L\gamma e^{\alpha t}(1+ \alpha t - e^{\alpha t})-L\alpha^{2}\beta t e^{\alpha t}<0,\]
			as $e^{\alpha t} > 1 + \alpha t$ for all $t>0$. Hence $h$ is decreasing function. Since $h(0)=0$, $h(t)<0$ for $t>0$ and  $H$ is  decreasing. 
		\end{proof}
		
		It is easy to check that 
		$\lim_{t \to 0} H(t) = 1$ and $\lim_{t \to +\infty} H(t) = -\infty$. We denote by $\tau_1$ and $\tau_2$ the unique positive numbers such that $H(\tau_1)=0$ and $H(\tau_2)=\frac{1}{2}$, respectively.
		
		\begin{lemma}\label{Bound_Mz}
			For every $t\in (0,\tau_{1})$, we have
			$$M_{z}(t)\leq \frac{\gamma\|\nabla f(z)\|}{\alpha H(t)}.$$
		\end{lemma}
		
		\begin{proof}
			If $0<u\leq t$, using (\ref{Int}), (\ref{IJe}) and Lemma \ref{IJ}, we obtain
			\begin{align*}
				\frac{\|\dot{x}(u)\|}{1-e^{-\alpha u}}
				&\leq \frac{\|I_{z}(u)+J_{z}(u)\|}{e^{\alpha u}-1}+\frac{\gamma}{\alpha}\|\nabla f(z)\|	\\
				&\leq L M_{z}(u)\left[\frac{\gamma  \left(\frac{1}{\alpha}u e^{\alpha u}-\frac{2}{\alpha^{2}}e^{\alpha u}+\frac{2}{\alpha^{2}}+\frac{u}{\alpha}\right)+\beta  \left(\frac{1}{\alpha} e^{\alpha u}-\frac{1}{\alpha}-u\right) }{e^{\alpha u}-1}\right] +\frac{\gamma}{\alpha}\|\nabla f(z)\|\\
				&\leq L M_{z}(u)\left[\frac{u \left(\frac{\gamma}{\alpha} e^{\alpha u}+\frac{\gamma}{\alpha}-\beta\right)+ (e^{\alpha u}-1) \left(\frac{\beta}{\alpha} -\frac{2\gamma}{\alpha^{2}}\right) }{e^{\alpha u}-1}\right]+\frac{\gamma}{\alpha}\|\nabla f(z)\|. 
			\end{align*}   
			The right-hand side is non-decreasing in $u$. By taking supremum over $u\in (0,t]$, we get
			\[ M_{z}(t)\leq L M_{z}(t)\left[\frac{t \left(\frac{\gamma}{\alpha} e^{\alpha t}+\frac{\gamma}{\alpha}-\beta\right)+ (e^{\alpha t}-1) \left(\frac{\beta}{\alpha} -\frac{2\gamma}{\alpha^{2}}\right) }{e^{\alpha t}-1}\right] +\frac{\gamma}{\alpha}\|\nabla f(z)\|.\]
			The result is obtained by arranging the terms, and using (\ref{H}).
		\end{proof}
		
		Lemmas \ref{IJ} and \ref{Bound_Mz} together yield:
		
		\begin{corollary} \label{estimate}
			For every $t\in (0,\tau_{1})$, we have
			\begin{eqnarray*}
				\|I_{z}(t)+J_{z}(t)\|
				& \leq & (e^{\alpha t}-1)\left[ \frac{1-H(t)}{H(t)}\right] \frac{\gamma}{\alpha} \|\nabla f(z)\|, \\
				\|\gamma(\nabla f(x(t))-\nabla f(z))+\beta \nabla^{2}f(x(t))\dot{x}(t) \| 
				& \leq & \left[\gamma t +\left(\frac{\gamma}{\alpha}-\beta\right)(e^{-\alpha t}-1)\right]\frac{L\gamma \|\nabla f(z)\|}{\alpha H(t)}.
			\end{eqnarray*}   
		\end{corollary}
		
		\subsection{Estimations for the speed restarting time}
		
		We now establish upper and lower bounds for the restarting time, depending on $\alpha$, $\beta$, $\gamma$ and $L$, and {\it not} on the initial condition $z$. We begin by bounding the inner product involved in the definition of the speed restarting time.
		
		\begin{lemma} \label{L:bound_acceleration}
			Let $z\notin \operatorname{argmin}(f)$ and let $x$ be the solution of (\ref{ORI}) with initial conditions $x(0)=z$ and $\dot{x}(0)=0$. For every $t\in (0, \tau_{1})$, we have	
			\[ \langle \dot{x}(t),\ddot{x}(t)\rangle\geq \frac{\gamma^{2}(e^{\alpha t}-1)\|\nabla f(z)\|^{2}}{\alpha e^{2\alpha t} H(t)^{2}}G(t),\]
			where
			\begin{equation}\label{eq:def_G}
				G(t)=1-\frac{Lte^{\alpha t}(\frac{\gamma}{\alpha}e^{\alpha t}+\frac{\gamma}{\alpha}-\beta)}{e^{\alpha t}-1}+\frac{3L\gamma}{\alpha^{2}}e^{\alpha t}-\frac{2L\beta}{\alpha}e^{\alpha t}-\frac{L\gamma}{\alpha^{2}}+\frac{L\beta}{\alpha}-\frac{L\gamma}{\alpha}te^{\alpha t}.
			\end{equation}
		\end{lemma}
		\begin{proof}
			From \eqref{Int} and \eqref{IJe}, we obtain
			\begin{equation}\label{dotx}
				\dot{x}(t)=-\frac{1}{e^{\alpha t}}(I_{z}(t)+J_{z}(t))-\frac{\gamma(e^{\alpha t}-1)\nabla f(z)}{\alpha e^{\alpha t}}.
			\end{equation}
			Also,
			\[\frac{d}{dt}\left(\frac{1}{e^{\alpha t}}(I_{z}(t)+J_{z}(t))\right)=-\frac{\alpha}{e^{\alpha t}}(I_{z}(t)+J_{z}(t))+\gamma(\nabla f(x(t))-\nabla f(z))+\beta\nabla^{2}f(x(t))\dot{x}(t).\]
			Then,
			\[\ddot{x}(t)= \frac{\alpha}{e^{\alpha t}}(I_{z}(t)+J_{z}(t))-\gamma(\nabla f(x(t))-\nabla f(z))-\beta\nabla^{2}f(x(t))\dot{x}(t)-\frac{\gamma\nabla f(z)}{e^{\alpha t}}.\]
			Let us define
			\[A(t)=\frac{\alpha}{e^{\alpha t}}(I_{z}(t)+J_{z}(t))-\frac{\gamma\nabla f(z)}{e^{\alpha t}}, \quad B(t)=\gamma(\nabla f(x(t))-\nabla f(z))+\beta\nabla^{2}f(x(t))\dot{x}(t).\] 
			By using the triangle inequality we have
			\begin{equation}\label{eq:innerAB}
				\langle \dot{x}(t),\ddot{x}(t)\rangle=\langle\dot{x}(t),A(t)\rangle-
				\langle \dot{x}(t),B(t)\rangle  \geq \langle \dot{x}(t),A(t)\rangle-\|\dot{x}(t)\|\|B(t)\|.
			\end{equation}
			The first term can be bounded from below by using Corollary \ref{estimate}, giving
			
			\begin{align*}
				&\langle\dot{x}(t),A(t)\rangle \\
				=& -\left\langle \frac{1}{e^{\alpha t}}(I_{z}(t)+J_{z}(t))+\frac{\gamma}{\alpha}(1-e^{-\alpha t})\nabla f(z), \frac{\alpha}{e^{\alpha t}}(I_{z}(t)+J_{z}(t))-\frac{\gamma \nabla f(z)}{e^{\alpha t}} \right\rangle \\
				=& -\frac{\alpha}{e^{2\alpha t}}\|I_{z}(t)+J_{z}(t)\|^{2}+\frac{\gamma}{e^{2\alpha t}}\langle I_{z}(t)+J_{z}(t), \nabla f(z) \rangle\\
				&-\frac{\gamma (e^{\alpha t}-1)}{e^{2\alpha t}}\langle I_{z}(t)+J_{z}(t), \nabla f(z) \rangle+\frac{\gamma^{2}(e^{\alpha t}-1)\|\nabla f(z)\|^{2}}{\alpha e^{2\alpha t}}\\
				\geq& -\frac{\gamma^{2}(e^{\alpha t}-1)^{2}\|\nabla f(z)\|^{2}}{\alpha e^{2\alpha t}}\left[\frac{1-H(t)}{H(t)}\right]^{2}-\frac{\gamma^{2}(e^{\alpha t}-2)(e^{\alpha t}-1)\|\nabla f(z)\|^{2}}{\alpha e^{2\alpha t}}\left[\frac{1-H(t)}{H(t)}\right]\\
				&+\frac{\gamma^{2}(e^{\alpha t}-1)\|\nabla f(z)\|^{2}}{\alpha e^{2\alpha t}}\\
				=&\frac{\gamma^{2}(e^{\alpha t}-1)\|\nabla f(z)\|^{2}}{\alpha e^{2\alpha t}}\left(- (e^{\alpha t}-1)\left[\frac{1-H(t)}{H(t)}\right]^{2}-(e^{\alpha t}-2)\left[\frac{1-H(t)}{H(t)}\right]+1  \right)\\
				=&\frac{\gamma^{2}(e^{\alpha t}-1)\|\nabla f(z)\|^{2}}{\alpha e^{2\alpha t}H(t)^{2}}\left(H(t)^{2}-(e^{\alpha t}-1)(1-H(t))^{2}-(e^{\alpha t}-2)H(t)(1-H(t)) \right)\\
				=&\frac{\gamma^{2}(e^{\alpha t}-1)\|\nabla f(z)\|^{2}}{\alpha e^{2\alpha t}H(t)^{2}}\left(e^{\alpha t}H(t)-(e^{\alpha t}-1)\right).
			\end{align*}
			For the second term, from (\ref{dotx}) and Corollary \ref{estimate}, we can observe that
			\begin{align*}
				\|\dot{x}(t)\|&\leq \frac{1}{e^{\alpha t}}\|I_{z}(t)+J_{z}(t)\|+\frac{\gamma(e^{\alpha t}-1)\|\nabla f(z)\|}{\alpha e^{\alpha t}}\\
				&\leq \frac{\gamma(e^{\alpha t}-1)\|\nabla f(z)\|}{\alpha e^{\alpha t}}\left[\frac{1-H(t)}{H(t)}\right]+\frac{\gamma(e^{\alpha t}-1)\|\nabla f(z)\|}{\alpha e^{\alpha t}}\\
				&=\frac{\gamma(e^{\alpha t}-1)\|\nabla f(z)\|}{\alpha e^{\alpha t}H(t)},
			\end{align*}
			and
			\begin{align*}
				\|B(t)\|&\leq\gamma \|\nabla f(x(t))-\nabla f(z)\|+\beta \nabla^{2}f(x(t))\dot{x}(t)\\
				&\leq \frac{\gamma\|\nabla f(z)\|}{H(t)e^{\alpha t}}\left[\paren{\frac{L\gamma}{\alpha^{2}}-\frac{L\beta}{\alpha}}(1-e^{\alpha t})+\frac{L\gamma}{\alpha} te^{\alpha 
					t}\right],
			\end{align*}
			thus
			\[\|\dot{x}(t)\|\|B(t)\|\leq
			\frac{\gamma^{2}(e^{\alpha t}-1)\|\nabla f(z)\|^{2}}{\alpha e^{2\alpha t}H(t)^{2}}\left[\left(\frac{L\gamma}{\alpha^{2}}-\frac{L\beta}{\alpha}\right)(1-e^{\alpha t})+\frac{L\gamma}{\alpha} te^{\alpha 
				t}\right].\]
			Using the obtained inequalities in \eqref{eq:innerAB}, we obtain
			
			\begin{align*}    
				\langle \dot{x}(t),\ddot{x}(t)\rangle  &\geq\frac{\gamma^{2}(e^{\alpha t}-1)\|\nabla f(z)\|^{2}}{\alpha e^{2\alpha t}H(t)^{2}}\left(e^{\alpha t}H(t)-e^{\alpha t}+1-\left(\frac{L\gamma}{\alpha^{2}}-\frac{L\beta}{\alpha}\right)(1-e^{\alpha t})-\frac{L\gamma}{\alpha} te^{\alpha 
					t}\right),  
			\end{align*}
			which is the desired inequality.
		\end{proof}
		
		\begin{proposition}\label{function:G(t)}
			The function $G$ has a unique zero on $(0,\tau_1)$, which we denote by $\tau_3$.
		\end{proposition}
		
		\begin{proof} 
			The derivative of $G$ is given by $G'(t)=Le^{\alpha t}(e^{\alpha t}-1)^{-2}g(t)$, where
			\[g(t)=\frac{\gamma}{\alpha}e^{2\alpha t}-2\beta e^{2\alpha t}-2\gamma te^{2\alpha t}+5\beta e^{\alpha t}+4\gamma te^{\alpha t}-\frac{4\gamma}{\alpha}e^{\alpha t}+\frac{3\gamma}{\alpha}-3\beta-\alpha\beta t.\]
			The sign of $g$ determines that of $G'$. Observe that
			\begin{eqnarray*}
				g'(t) & = & -4\alpha\beta e^{2\alpha t}(\beta + \gamma t) + \alpha e^{\alpha t}(5\beta + 4\gamma t) - \alpha\beta, \\
				g''(t) & = & -\alpha^2 e^{2\alpha t}(3\beta + 4)-\alpha e^{\alpha t}(e^{\alpha t}-1)\paren{4\alpha\gamma t +5\alpha\beta+4\gamma}.
			\end{eqnarray*}
			Since $g'(0)=0$ and $g''(t)<0$ for all $t>0$, we deduce that $g'(t)<0$ (and so $g$ is decreasing) for all $t>0$. Since $g(0)=0$, $g$ must be negative on $(0,\infty)$, and so $G$ is decreasing. We conclude by observing that $\lim_{t \to 0} G(t) =1$ and $\lim_{t \to +\infty} G(t) =-\infty$.
		\end{proof}
		
		Lemma \ref{L:bound_acceleration} then gives:
		
		\begin{corollary}\label{infc}
			For every $z\notin \operatorname{argmin}(f)$, we have $\tau_{3}\leq T(z)$.
		\end{corollary}
		
		
		Next, we derive an upper bound for the speed restarting time. In what follows, we assume that $f$ satisfies condition (\ref{PL}) with $\mu>0$. 
		
		\begin{proposition}\label{upper bound}
			Let $z\notin \operatorname{argmin}(f)$. For each $\tau\in (0,\tau_{2})\cap \left(0,T(z)\right]$, we have 
			\[T(z)\leq \tau+\frac{\alpha}{2\mu\gamma(1-e^{-\alpha \tau})^{2}\Psi(\tau) }, \quad where\quad \Psi (\tau)=\left[2-\frac{1}{H(\tau)}\right]^{2}.\]
		\end{proposition}
		
		\begin{proof}
			In view of (\ref{Int}) and (\ref{IJe}) and Corollary \ref{estimate}, we have
			\[\left\|\dot{x}(\tau)+\frac{\gamma}{\alpha}(1-e^{-\alpha \tau})\nabla f(z)\right\|=\frac{1}{e^{\alpha\tau}}\|I_{z}(t)+J_{z}(t)\|\leq \frac{\gamma}{\alpha}(1-e^{-\alpha\tau})\left[\frac{1-H(t)}{H(t)}\right]\|\nabla f(z)\|.\]
			The reverse triangle inequality gives
			\begin{equation}\label{triangle}
				\begin{array}{ll}
					\|\dot{x}(\tau)\| &\geq \frac{\gamma}{\alpha}(1-e^{-\alpha \tau})\|\nabla f(z)\|-\frac{\gamma}{\alpha}(1-e^{-\alpha \tau})\left[\frac{1-H(\tau)}{H(\tau)}  \right]\|\nabla f(z)\| \\
					&=\frac{\gamma}{\alpha}(1-e^{-\alpha \tau})\left[2-\frac{1}{H(\tau)}\right]\|\nabla f(z)\|.
				\end{array}  
			\end{equation}
			Observe that the expression on the last equality is positive since $\tau\in(0,\tau_{2})$. Taking $t\in [\tau, T(z)]$, and as $\|\dot{x}(t)\|^{2}$ increases in $[0,T(z)]$, \eqref{increase} gives
			
			\begin{align*}
				\frac{d}{dt}f(x(t)) &\leq -\frac{\alpha}{\gamma}\|\dot{x}(t)\|^{2}\\
				&\leq -\frac{\alpha}{\gamma}\|\dot{x}(\tau)\|^{2}  \\
				& \leq- \frac{\alpha}{\gamma}\left[\frac{\gamma}{\alpha}(1-e^{-\alpha\tau})\left[2-\frac{1}{H(\tau)}\right]\|\nabla f(z)\|\right]^{2}\\
				&=-\frac{\gamma}{\alpha}(1-e^{-\alpha\tau})^{2}\Psi(\tau)\|\nabla f(z)\|^{2}.
			\end{align*}
			
			Integrating over $[\tau,T(z)]$, we get
			\[
			f(x(T(z))-f(x(\tau))\leq -\left[ \frac{\gamma}{\alpha}(1-e^{-\alpha\tau})^{2}\Psi(\tau)\|\nabla f(z)\|^{2}\right](T(z)-\tau) 
			.\]
			It follows that 
			$$\frac{\gamma}{\alpha}(1-e^{-\alpha\tau})^{2}\Psi(\tau)\|\nabla f(z)\|^{2}(T(z)-\tau)\leq f(x(\tau))-f(x(T(z))  
			\leq f(z)-f^{*} 
			\leq\frac{1}{2\mu}\|\nabla f(z)\|^{2},
			$$
			and the result follows by arranging the terms.
		\end{proof}	
		
		We denote
		$$\tau^*=\operatorname{argmin}\left\{\tau+\frac{\alpha}{2\mu\gamma(1-e^{-\alpha \tau})^{2}\Psi(\tau)}:\tau\in (0,\tau_{2})\cap \left(0,T(z)\right] \right\}.$$
		
		
		\subsection{Function Value Decrease}\label{sec:results}
		The next result provides the ratio at which the function values reduce at each interval built by the restart criteria. 
		\begin{proposition}\label{fdecrease}
			Let $z\notin \operatorname{argmin}(f)$, and let $x$ be the solution of (\ref{ORI}) with initial condition $x(0)=z$ and $\dot{x}(0)=0$. Let $f$ satisfy (\ref{PL}) with $\mu>0$. For each $\tau\in (0,\tau_{2})\cap (0,T(z)]$, we have 
			\[f(x(t))-f^{*}\leq \left[1-\frac{2\mu\gamma}{\alpha}\Psi(\tau)\left(\tau+\frac{2}{\alpha}e^{-\alpha\tau}-\frac{1}{2\alpha}e^{-2\alpha \tau}-\frac{3}{2\alpha}\right)\right](f(z)-f^{*})\]
			for every $t\in \left[\tau, T(z)\right]$.
		\end{proposition}
		
		\begin{proof}
			Let $s\in (0,\tau)$, using \eqref{increase} and (\ref{triangle}), we obtain
			\begin{align*}
				\frac{d}{ds}f(x(s)) &\leq- \frac{\alpha}{\gamma}\|\dot{x}(s)\|^{2}  \\
				& \leq -\frac{\alpha}{\gamma}\left[\frac{\gamma}{\alpha}(1-e^{-\alpha s})\left[2-\frac{1}{H(s)}\right]\|\nabla f(z)\|\right]^{2}\\
				&
				\leq-\frac{\gamma}{\alpha}(1-e^{-\alpha s})^{2}\Psi(\tau)\|\nabla f(z)\|^{2},
			\end{align*}
			where the last inequality is obtained since $H(t)$ is decreasing in $(0,\tau_{1})$, which contains $(0,\tau)$. 
			Integrating over $(0,\tau)$, using (\ref{PL}), we obtain
			\begin{align*}
				f(x(\tau))-f^{*} & \leq f(z)-f^{*}-\int_{0}^{\tau}\frac{\gamma}{\alpha}(1-e^{-\alpha s})^{2}\Psi(\tau)\|\nabla f(z)\|^{2}ds \\
				& \leq f(z)-f^{*}-\frac{\gamma}{\alpha}\Psi(\tau)\|\nabla f(z)\|^{2}\int_{0}^{\tau}(1-2e^{-\alpha s}+e^{-2\alpha s})ds\\
				&=f(z)-f^{*}-\frac{\gamma}{\alpha}\Psi(\tau)\|\nabla f(z)\|^{2}\left(\tau+\frac{2}{\alpha}e^{-\alpha\tau}-\frac{1}{2\alpha}e^{-2\alpha \tau}-\frac{3}{2\alpha}\right)\\
				&\leq \left[1-\frac{2\mu\gamma}{\alpha}\Psi(\tau)\left(\tau+\frac{2}{\alpha}e^{-\alpha\tau}-\frac{1}{2\alpha}e^{-2\alpha \tau}-\frac{3}{2\alpha}\right)\right](f(z)-f^{*}).
			\end{align*}
			As $t\in \left[\tau, T(z)\right]$, Remark \ref{increase}, gives $f(x(t))\leq f(x(\tau))$ and the result is obtained.
		\end{proof}

		\section{Some Numerical Illustrations}
		\label{sec:numerical}

		\if{
			
			\subsection{A simple  example: explicit computation of the restart time}
			Consider \eqref{ORI_intr} with $x(0)=x_0 \in \R$, $\dot{x}(0)=0$. Let $f:\R \to \R$ be defined by $f(x) = \frac{1}{2}x^2$. We obtain the system
			\begin{equation}\label{eq:sist_toy}
				\ddot{x}(t) + (\alpha+\beta)\dot{x}(t)+\gamma x(t) = 0.
			\end{equation}
			The solutions of \eqref{eq:sist_toy} can be computed depending on the relation between the parameters. Then, the product $\dot{x}(t)\ddot{x}(t)$ that defines the speed restart time can be used to explicitly find the speed restart time. This time can be compared with the lower bound given by the positive root of $G(t)$ defined in \eqref{eq:def_G}. We study three different cases. 
			
			\subsubsection*{Case 1: $\Delta = (\alpha+\beta)^2 - 4\gamma >0$}
			Here, the characteristic equation has two distinct real solutions, given by
			\[\kappa_\pm = -\dfrac{(\alpha+\beta)}{2}\pm \dfrac{\sqrt{\Delta}}{2}.\]
			Notice that as $\alpha$ and $\beta$ are positive, then the two solutions are negative. The general solution of \eqref{eq:sist_toy} is given by
			\[x(t) = C_1 e^{\kappa_{+}t} + C_2 e^{\kappa_{-}t}.\]
			Imposing the initial conditions, $C_1 + C_2 = x_0$ and $\kappa_{+}C_1 + \kappa_{-}C_2 =0$, we get
			\[C_1 = -\dfrac{x_0\kappa_{-}}{\sqrt{\Delta}}, \quad C_2 = \dfrac{x_0\kappa_{+}}{\sqrt{\Delta}}.\]
			Computing the derivatives, we obtain
			\begin{align*}
				\dot{x}(t) &= C_1\kappa_{+}e^{\kappa_{+}t} + C_2\kappa_{-}e^{\kappa_{-}t}, \\
				\ddot{x}(t) &= C_1\kappa_{+}^2e^{\kappa_{+}t} + C_2\kappa_{-}^2e^{\kappa_{-}t}.
			\end{align*}
			Thus, 
			\[\dot{x}(t)\ddot{x}(t) = \delta e^{-(\alpha+\beta)t}\paren{\kappa_{+}e^{\sqrt{\Delta}t} + \kappa_{-}e^{-\sqrt{\Delta}t} - \dfrac{\gamma}{\delta}C_1C_2(\alpha+\beta)},\]
			where $\delta = \kappa_{+}^2C_1^2 = \kappa_{-}^2C_2^2$ (we have also used that $\kappa_{+}\kappa_{-}=\gamma$). From the expressions for $C_1$ and $C_2$, we obtain
			\[\dot{x}(t)\ddot{x}(t) = \delta e^{-(\alpha+\beta)t}\paren{\kappa_{+}e^{\sqrt{\Delta}t} + \kappa_{-}e^{-\sqrt{\Delta}t} +\alpha+\beta}.\]
			Then, the speed restart time will be given by the study of the zeros of  the function 
			\[\varphi(t) =\kappa_{+}e^{\sqrt{\Delta}t} + \kappa_{-}e^{-\sqrt{\Delta}t} +\alpha+\beta. \]
			The equation $\varphi(t)=0$ can be solved, giving that the function has two zeros, one in $t=0$, and the other one in
			\begin{equation}\label{eq:restart_explicit1}
				t_{\text{sr}} = \dfrac{1}{\sqrt{\Delta}}\ln\paren{\dfrac{\kappa_{-}}{\kappa_{+}}},
			\end{equation}
			when the speed restart occurs. Figure \ref{fig:delta_pos} shows the function $\varphi(t)$ and the theoretical function $G(t)$ giving the lower bound, in the case $\alpha=3$, $\beta=1$, $\gamma=1$. It can be observed the  actual restart time $t_{\text{sr}}$ given by \eqref{eq:restart_explicit1} and its approximation given by $G(t)=0$. 
			
			\subsubsection*{Case 2: $\Delta = (\alpha+\beta)^2 - 4\gamma =0$}
			In this case, the general solution for \eqref{eq:sist_toy} is 
			\[x(t) = C_1 e^{-\frac{(\alpha+\beta)}{2}t} + C_2 te^{-\frac{(\alpha+\beta)}{2}t}.\]
			Imposing the initial conditions, we obtain that 
			\[C_1 = x_0, \quad C_2 = \dfrac{(\alpha+\beta)}{2}x_0.\]
			Computing the derivatives, we obtain 
			\begin{align*}
				\dot{x}(t) &= -x_0\gamma t e^{-\frac{(\alpha+\beta)}{2}t}, \\
				\ddot{x}(t) &= -x_0\gamma e^{-\frac{(\alpha+\beta)}{2}t}\paren{1 - \dfrac{(\alpha+\beta)}{2}t}.	
			\end{align*}
			Then, it is easy to see that $\dot{x}(t)\ddot{x}(t)$ is zero at $t=0$ and in
			\[t_{\text{sr}} = \dfrac{2}{\alpha+\beta},\]
			when the speed restart occurs. Figure \ref{fig:delta_zero} compares the function $\eta(t) = 1 - \frac{(\alpha+\beta)}{2}t$ and $G(t)$, in the case $\alpha=3$, $\beta=1$, $\gamma=4$. 
			
			\subsubsection*{Case 3: $\Delta = 4\gamma -(\alpha+\beta)^2  >0$}
			In this case, the general solution for \eqref{eq:sist_toy} is
			\[x(t) = e^{-\frac{(\alpha+\beta)}{2}t}\paren{C_1\cos\paren{\dfrac{\sqrt{\Delta}}{2}t} + C_2 \sin\paren{\dfrac{\sqrt{\Delta}}{2}t}}.\]
			Imposing the initial conditions we obtain that 
			\[C_1 = x_0, \quad C_2 = \dfrac{(\alpha+\beta)}{\sqrt{\Delta}}x_0.\]
			Computing the derivatives, we obtain 
			\begin{align*}
				\dot{x}(t) &= -\dfrac{2\gamma x_0}{\sqrt{\Delta}}e^{-\frac{(\alpha+\beta)}{2}t}\sin\paren{\dfrac{\sqrt{\Delta}}{2}t}, \\
				\ddot{x}(t) &= e^{-\frac{(\alpha+\beta)}{2}t}\paren{-x_0\gamma\cos\paren{\dfrac{\sqrt{\Delta}}{2}t} + \gamma x_0 \dfrac{(\alpha+\beta)}{\sqrt{\Delta}} \sin\paren{\dfrac{\sqrt{\Delta}}{2}t}}.
			\end{align*}
			Then, 
			\[\dot{x}(t)\ddot{x}(t) =2\dfrac{\gamma^2 x_0^2}{\sqrt{\Delta}}e^{-(\alpha+\beta)t}\paren{\sin\paren{\dfrac{\sqrt{\Delta}}{2}t}\cos\paren{\dfrac{\sqrt{\Delta}}{2}t}-\dfrac{(\alpha+\beta)}{\sqrt{\Delta}}\sin^2\paren{\dfrac{\sqrt{\Delta}}{2}t}}.\]
			We study the function
			\[\phi(t) = \sin\paren{\dfrac{\sqrt{\Delta}}{2}t}\paren{\cos\paren{\dfrac{\sqrt{\Delta}}{2}t}-\dfrac{(\alpha+\beta)}{\sqrt{\Delta}}\sin\paren{\dfrac{\sqrt{\Delta}}{2}t}}.\]
			It is easy to check that $\phi(0)=0$, and computing its derivative we get
			\[\phi'(t) = \dfrac{1}{2}\paren{\sqrt{\Delta}\cos\paren{\sqrt{\Delta} t}-(\alpha+\beta)\sin\paren{\sqrt{\Delta} t}}.\]
			The zeros of the derivative can be computed and as $\phi'(0)>0$, the derivative starts being positive after $t=0$ and hence, $\phi(t)$ starts increasing after $t=0$ . Then, $\phi(t)$ is a positive function between $t=0$ and its first positive zero, when the restart occurs. Notice that the function will be zero if 
			\[\sin\paren{\dfrac{\sqrt{\Delta}}{2}t}=0\Longleftrightarrow t = \dfrac{2k\pi}{\sqrt{\Delta}},\]
			being $k$ an integer. On the other hand, the function is zero if
			\[\tan\paren{\dfrac{\sqrt{\Delta}}{2}t} =\dfrac{\sqrt{\Delta} }{(\alpha+\beta)}
			\Longleftrightarrow t = \dfrac{2}{\sqrt{\Delta}}\arctan\paren{\dfrac{\sqrt{\Delta} }{(\alpha+\beta)}} + \dfrac{2k\pi}{\sqrt{\Delta}},\]
			with $k$ an integer. Then, the first positive zero of the function and hence, the speed restart time is attained at 
			\[t_{\text{sr}} = \min\set{\dfrac{2\pi}{\sqrt{\Delta}},\dfrac{2}{\sqrt{\Delta}}\arctan\paren{\dfrac{\sqrt{\Delta} }{(\alpha+\beta)}}} = \dfrac{2}{\sqrt{\Delta}}\arctan\paren{\dfrac{\sqrt{\Delta} }{(\alpha+\beta)}},\]
			since $\arctan$ is bounded from above by $\pi/2$. Figure \ref{fig:delta_neg} shows the function $\phi(t)$ and the theoretical lower bound given by $G(t)$, for the case $\alpha=3$, $\beta=1$, $\gamma=20$. 
			
			\begin{figure}[htbp]
				\centering
				\subfigure[$\alpha=3$, $\beta=1$, $\gamma=1$]{\includegraphics[width=0.32\textwidth]{Figures/toycase1.eps}\label{fig:delta_pos}}
				\subfigure[$\alpha=3$, $\beta=1$, $\gamma=4$]{\includegraphics[width=0.32\textwidth]{Figures/toycase2.eps}\label{fig:delta_zero}}
				\subfigure[$\alpha=3$, $\beta=1$, $\gamma=20$]{\includegraphics[width=0.32\textwidth]{Figures/toycase3.eps}\label{fig:delta_neg}}
				\caption{Depiction of the functions defining the speed restart in each case and function $G(t)$ giving the theoretical bound when $G(t)=0$. For each case we use $x_0=1$.}
				\label{fig:functions_toy}
			\end{figure}
			
			\subsubsection*{Gradient restart}
			Although it is beyond the scope of this paper, we can consider the gradient restart criteria or function values: there, the restart time happens at
			\[T(x_0) = \sup\set{t > 0 \,|\, \inner{\nabla f(x(u))}{\dot{x}(u)} < 0, \quad \forall u \in (0,t)}.\]
			In the context of the previous example, the product to analyze is simply $\dot{x}(t)x(t)$. Considering Case 3, we obtain that 
			\[\dot{x}(t)x(t) = -2\dfrac{\gamma x_0^2}{\sqrt{\Delta}}e^{-(\alpha+\beta)t} \sin\paren{\dfrac{\sqrt{\Delta}}{2}t}\paren{\cos\paren{\dfrac{\sqrt{\Delta}}{2}t}+\dfrac{(\alpha+\beta)}{\sqrt{\Delta}}\sin\paren{\dfrac{\sqrt{\Delta}}{2}t}}.\]
			The zeros of this expression are attained at
			\[t = \dfrac{2k\pi}{\sqrt{\Delta}} \; \lor \;  t = \dfrac{2}{\sqrt{\Delta}}\arctan\paren{-\dfrac{\sqrt{\Delta} }{(\alpha+\beta)}} + \dfrac{2k\pi}{\sqrt{\Delta}}, \quad k \in \mathbb{Z}. \]
			The first one of this times is when the gradient restart occurs, giving
			\[t_{\text{gr}} = \dfrac{2}{\sqrt{\Delta}}\arctan\paren{-\dfrac{\sqrt{\Delta} }{(\alpha+\beta)}} + \dfrac{2\pi}{\sqrt{\Delta}}.\]
			As $\arctan$ is an odd and bounded function , we obtain that 
			\[t_{\text{gr}} = \dfrac{2\pi}{\sqrt{\Delta}} - t_{\text{sr}} > \dfrac{\pi}{\sqrt{\Delta}} > t_{\text{sr}}. \]

		}\fi 
		
		
		In this section, we report the findings of some numerical experiments that illustrate how the convergence is improved by the speed restart scheme. \\
		
		We consider the function defined in \eqref{E:example_function}, with $n=3$ and $\rho=10$, namely
		\[f(x)=\frac{1}{2}(x_1^{2}+10 x_2^{2}+100x_3^{2}).\]
		
		\begin{example}[The solutions of \eqref{ORI_intr}] \label{EG:num_cont} 
			
			We take $\alpha=3$, $\beta\in\{0,6\}$ and $\gamma$, defined by \eqref{eq:gamma_example}, with $i=2$ and $\varepsilon\in\{10^{-1},10,10^2\}$, and display the evolution of the objective function values on the trajectories starting from $x(0)=(1,1,1)$ and initial velocity $\dot x(0)=(0,0,0)$, with and without restart. The results are shown in Figure \ref{fig:both with vare}. For the restarted trajectories, the results of approximating $f(x(t))\sim Ae^{-Bt}$, with $A,B>0$, via linear regression, are presented in Table \ref{tab:coef_regression_cont}.  \\
			
			\begin{figure}[htbp]
				\centering
				\subfigure[$\varepsilon=10^{-1}$]{\includegraphics[width=0.32\textwidth]{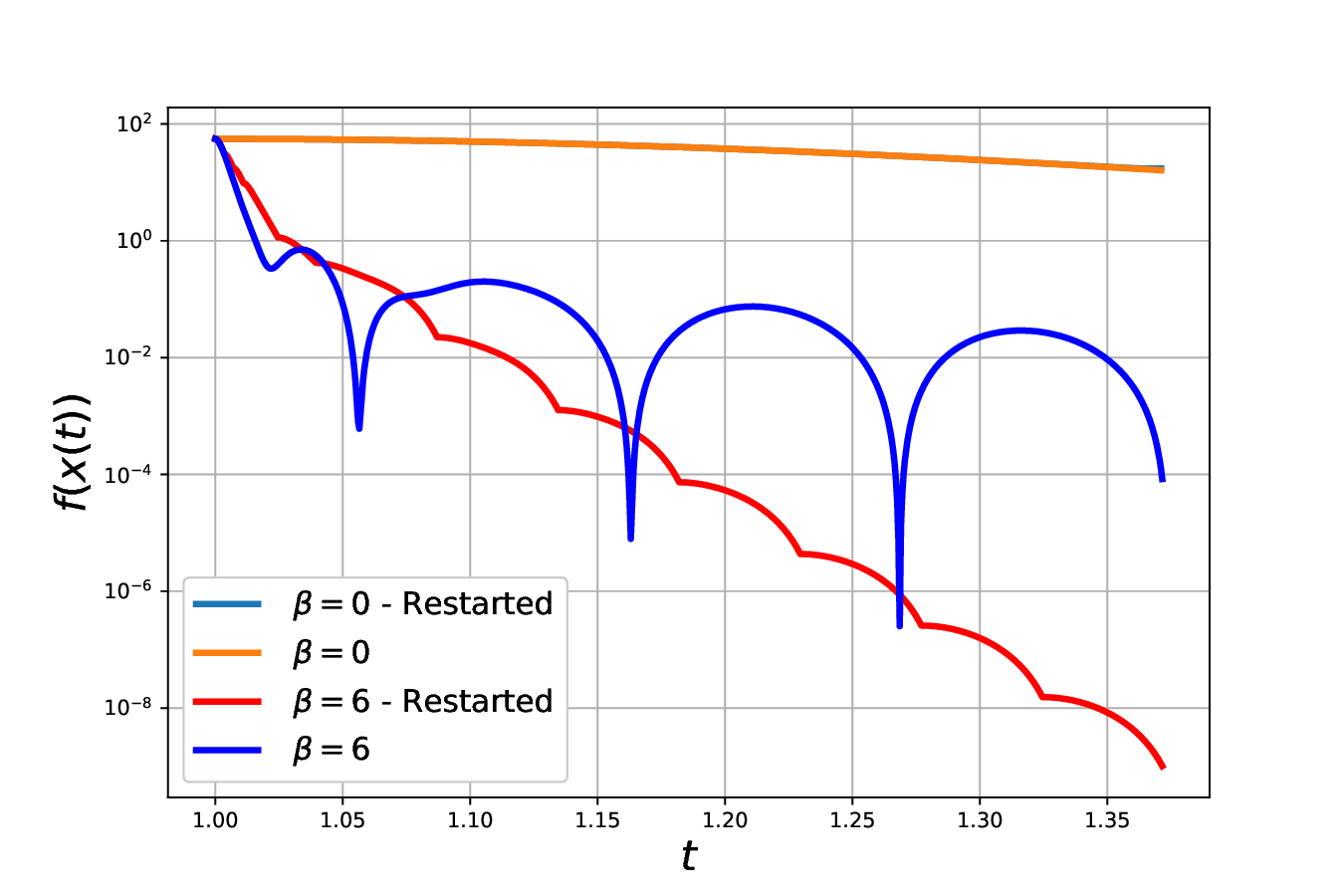}\label{fig:all0.1}}
				\subfigure[$\varepsilon=10$]{\includegraphics[width=0.32\textwidth]{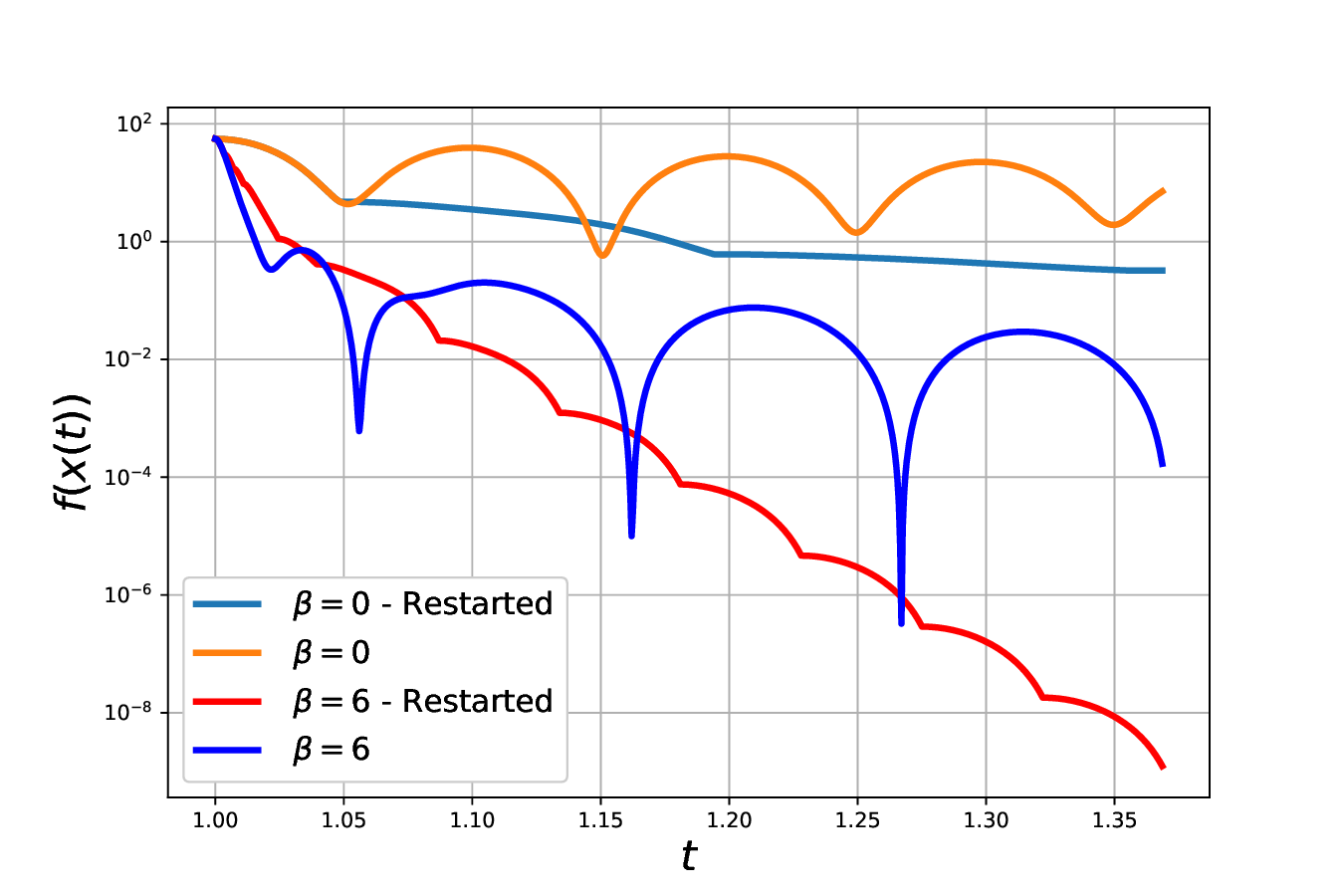}\label{fig:all10}}
				\subfigure[$\varepsilon=1000$]{\includegraphics[width=0.32\textwidth]{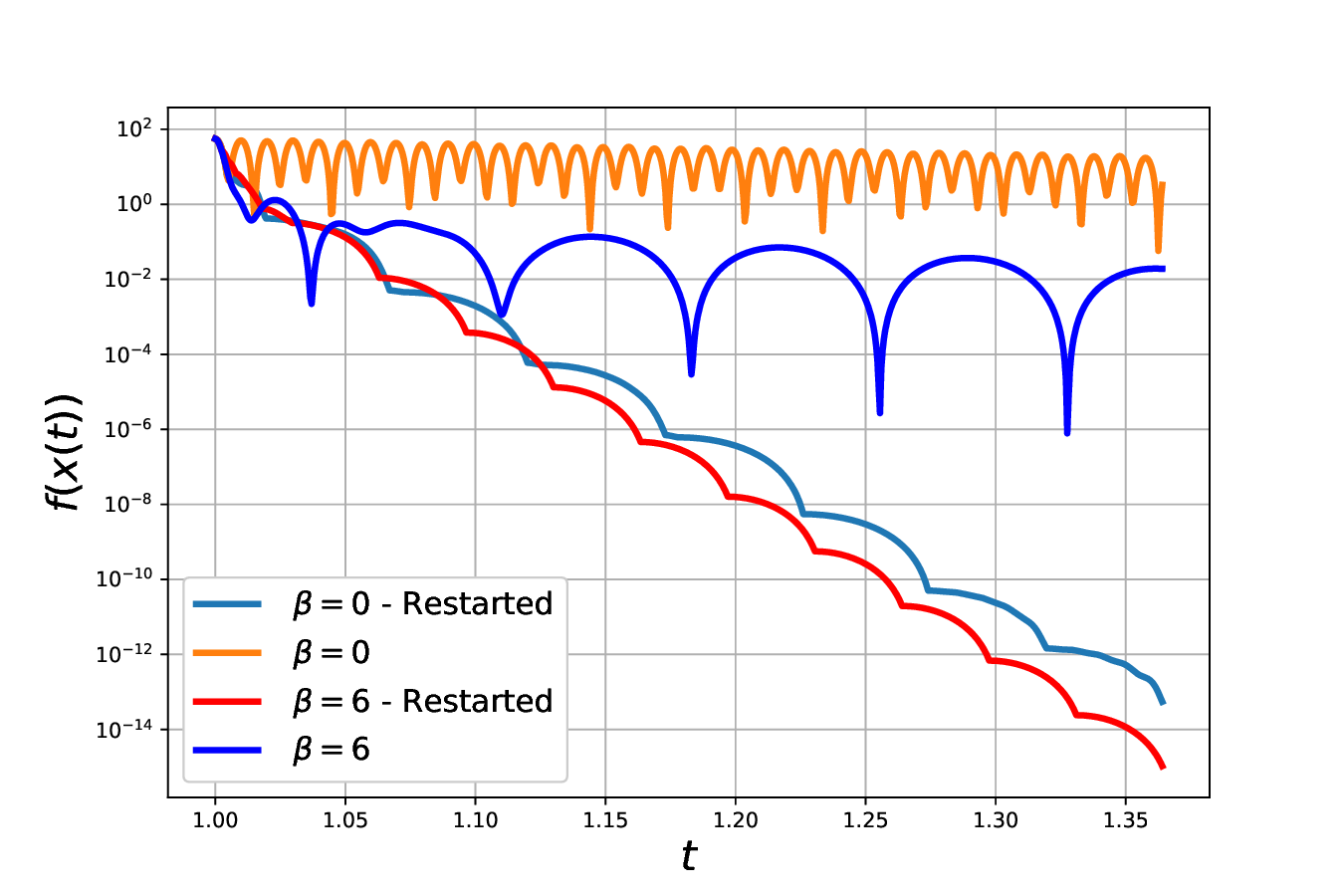}\label{fig:all1000}}
				\caption{Evolution of the objective function values along the trajectories in Example \ref{EG:num_cont}. In all cases, $\alpha=3$ and $\gamma$ is defined by \eqref{eq:gamma_example}, with $i=2$. The initial condition is $x(0)=(1,1,1)$ and $\dot x(0)=(0,0,0)$.
				}
				\label{fig:both with vare}
			\end{figure}
			
			\begin{table}[h]
				\centering
				\begin{tabular}{ccc|cc|cc}
					\toprule
					&\multicolumn{2}{c}{$\varepsilon=10^{-1}$} &\multicolumn{2}{c}{$\varepsilon=10$} 
					&\multicolumn{2}{c}{$\varepsilon=100$} \\
					\midrule
					& $\beta=0$ & $\beta=6$ & $\beta=0$ & $\beta=6$  & $\beta=0$ & $\beta=6$\\
					\midrule
					$A$ & 63.24   & 7.34  & 5.92    & 6.68 & 8.99   &   14.62\\
					$B$ & 2.99 & 59.72   & 6.62  & 59.14 & 88.51  & 101.57 \\
					\bottomrule
				\end{tabular}
				\caption{Coefficients in the linear regression, when approximating $f\big(x(t)\big)\sim Ae^{-Bt}$ in Example \ref{EG:num_cont}.}
				\label{tab:coef_regression_cont}
			\end{table}
			
			The performance is better when the Hessian-driven damping term is present, whether or not there is restarting. The restarted trajectories consistently perform better in the long run, although the non-restarted ones do attain some lower values at the beginning. This was expected, in view of the results of \cite{SBC16,MP23}. The regularity in the oscillatory behavior of \eqref{eq:din_alpha_beta} and \eqref{ORI_intr} seems to be {\it inherited} by the restarted trajectories, as can be seen from the way the restarting times are distributed. Another interesting phenomenon is that, although \eqref{eq:din_alpha_beta} oscillates more than \eqref{ORI_intr} for the highest value of $\varepsilon$ (whence that of $\gamma$), the corresponding restarting times are more spaced. Moreover, for each case, we consider the sequence $(T_{i+1} - T_{i})_{i \in \mathbb{N}}$ given by the restart times and we compute the mean value and the variance. The results are displayed on Table \ref{tab:mean_variance}. 
			
			\begin{table}[h]
				\centering
				\begin{tabular}{ccc|cc|cc}
					\toprule
					&\multicolumn{2}{c}{$\varepsilon=10^{-1}$} &\multicolumn{2}{c}{$\varepsilon=10$} 
					&\multicolumn{2}{c}{$\varepsilon=100$} \\
					\midrule
					& $\beta=0$ & $\beta=6$ & $\beta=0$ & $\beta=6$  & $\beta=0$ & $\beta=6$\\
					\midrule
					Mean & 7.01e-01   &  3.79e-02 &  3.70e-01 & 3.76e-02 &  3.39-e2  & 2.59e-2 \\
					Variance & 3.76e-01 & 2.85e-04 & 3.50e-03 & 2.79e-04 & 3.48e-4 & 1.51e-4\\
					\bottomrule
				\end{tabular}
				\caption{.}
				\label{tab:mean_variance}
			\end{table}

		\end{example}

		\begin{example}[A brief algorithmic exploration] \label{EG:num_algo}
			Several discretizations of \eqref{ORI_intr} with respect to time lead to first order algorithms that generate minimizing sequences for $f$. Although the main focus of this paper is not to analyze the numerical performance of algorithms, we present numerical results to illustrate the effect of including a speed restart routine on Algorithm \ref{algorithm} below, inspired by \cite{ACFR22}. \\
			
			\begin{algorithm}
				\caption{Inertial Gradient Algorithm with Hessian-Driven Damping and Three Constant Coefficients - Speed Restart Scheme}
				\begin{algorithmic}
					\STATE Given $x_{0},x_{1}\in\mathbb{R}^{n}$, $N$ and $h>0$. \\
					\textbf{for} $k=1$ to $N$ \textbf{do}:
					\STATE\quad 1. $y_{k}=x_{k}+(1-\alpha h)(x_{k}-x_{k-1})-\beta h(\nabla f(x_{k})-\nabla f(x_{k-1}))$.
					\STATE\quad 2. $x_{k+1}=y_{k}-\gamma h^{2}\nabla f(y_{k})$.\\
					\textbf{if} $\|x_{k+1}-x_{k}\|<\|x_{k}-x_{k-1}\|$ \textbf{then}
					\STATE\quad$x_{k}=x_{k-1}$\\
					\textbf{else}
					\STATE\quad$k=k+1$\\
					\textbf{end}\\
					\textbf{return} $x_{N}$.
				\end{algorithmic}
				\label{algorithm}
			\end{algorithm}
			
			We consider the function defined in \eqref{E:example_function} with $n=3$ and $\rho=10$. For the algorithm parameters we consider $\alpha=3$, $\beta=6$, $h=10^{-3}$ and $\gamma$ satisfying condition \eqref{eq:gamma_example} with $i=2$ and $\varepsilon\in\{10^{-1},10,10^2\}$. The initial point $x_1=x_{0}$ is generated randomly. Figure \ref{fig:discrete time} displays the function values obtained for Algorithm \ref{algorithm}, with and without speed restart routine. For reference, we also include the {\it warm start} variant proposed in \cite{MP23}, which consists in including one additional cycle at the beginning, ending in a function value restart, instead of a speed restart. For the speed restart scheme, we also perform an aprroximation of the function values as $f(x_k) \sim A e^{-Bk}$. Table \ref{tab:coef_regression_discrete} displays the values of $A$ and $B$ obtained for each value of $\varepsilon$ considered.  \\
			
			The results are similar to those obtained for the continuous case. One aspect to note is that the warm start does not seem to improve the slope of the corresponding plots, in contrast with the results in \cite{MP23} for \eqref{eq:AVD} and \eqref{eq:DIN-AVD}.    
		\end{example}
		
		\begin{figure}[htbp]
			\centering
			\subfigure[$\varepsilon=0.1$]{\includegraphics[width=0.32\textwidth]{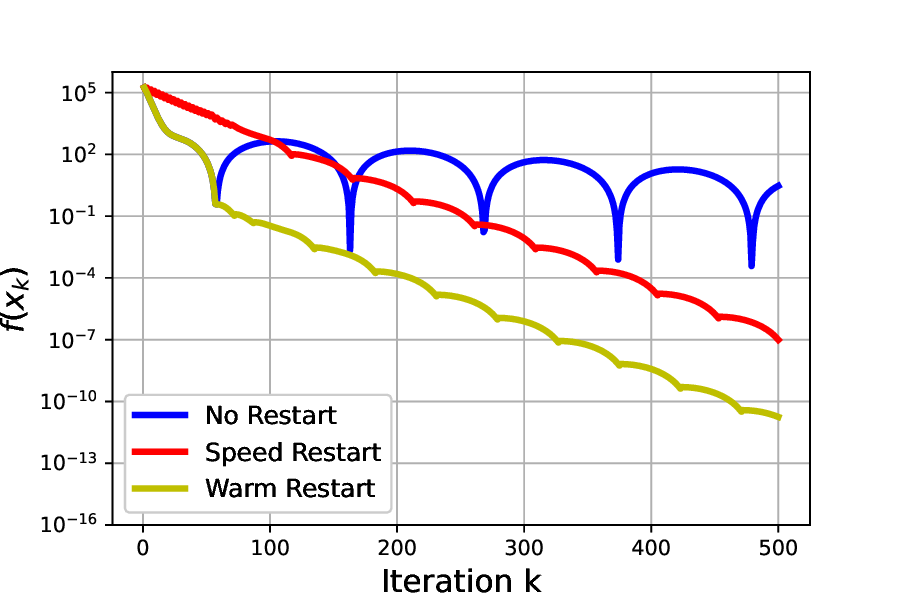}\label{fig:disc_alg0.1}}
			\subfigure[$\varepsilon=10$]{\includegraphics[width=0.32\textwidth]{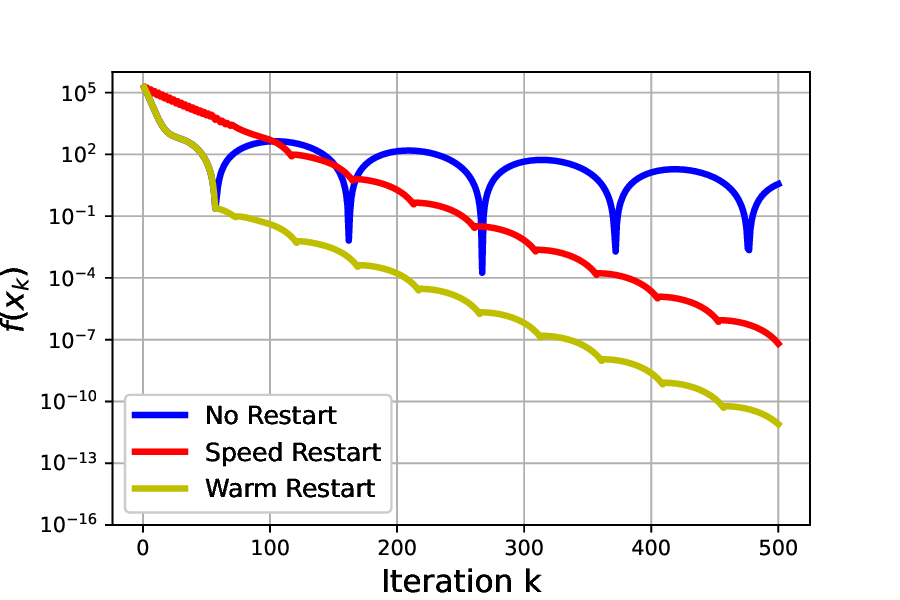}\label{fig:disc_alg10}}
			\subfigure[$\varepsilon=1000$]{\includegraphics[width=0.32\textwidth]{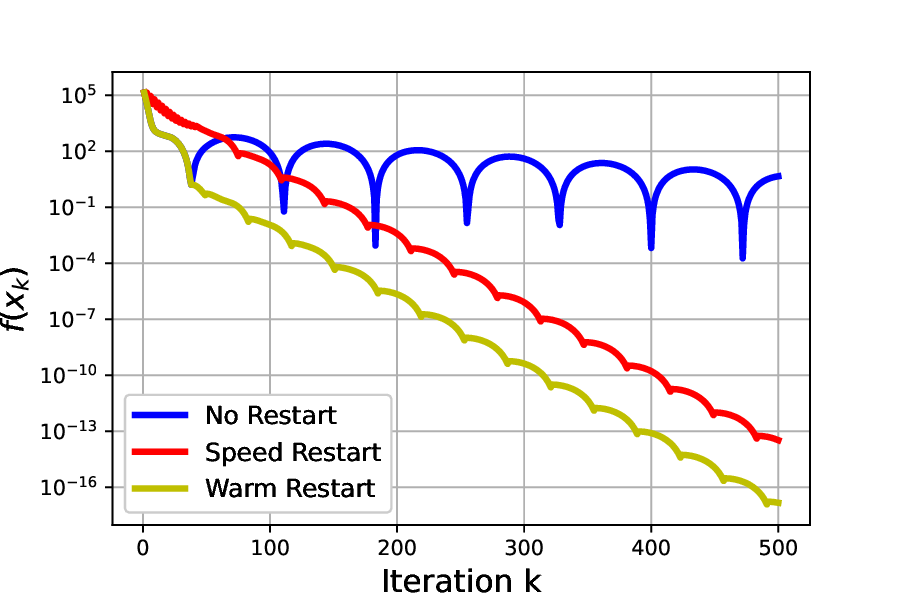}\label{fig:disc_alg1000}}
			\caption{Sequence of objective function values for Example \ref{EG:num_algo}. In all cases, $\alpha=3$ and $\gamma$ is defined by \eqref{eq:gamma_example}, with $i=2$. The initial condition is $x_0=x_1=(1,1,1)$.}
			\label{fig:discrete time}
		\end{figure}

		\begin{table}[h]
			\centering
			\begin{tabular}{cc|c|c}
				\toprule
				&$\varepsilon=10^{-1}$ &$\varepsilon=10$ 
				&$\varepsilon=100$ \\
				\midrule
				$A$ & 1.07e5   & 1.12e5  & 6.83e4\\
				$B$ & 5.46e-2 & 5.55e-2   & 8.52e-2 \\
				\bottomrule
			\end{tabular}
			\caption{Coefficients in the linear regression, when approximating $f\big(x_k\big)\sim Ae^{-Bt}$ in Example \ref{EG:num_algo}.}
			\label{tab:coef_regression_discrete}
		\end{table}

		\textit{Acknowledgments} This research benefited from the support of the FMJH Program Gaspard Monge for optimization and operations research and their interactions with data science. The first author was partially supported by the China Scholarship Council. The second author was partially supported by ANID-Chile grant Exploración 13220097, and Centro de Modelamiento Matemático (CMM) BASAL fund FB210005 for centers of excellence, from ANID-Chile.

		\bibliographystyle{abbrv}
		\bibliography{Bibliography}
		
	\end{document}